\documentclass[preprint, review, 10pt]{elsarticle}

\usepackage{amsmath, amssymb, amsfonts, amsthm}
\usepackage{hyperref}
\usepackage{graphicx}
\usepackage{mathrsfs}

\usepackage{color}
\usepackage{cases}

 \usepackage{diagbox}

\newtheorem{lemma}{Lemma}
\newtheorem{remark}{Remark}
\newtheorem{theorem}{Theorem}

\let\mr=\mathrm

\begin{document}

\begin{frontmatter}



\title{
Supercloseness of finite element method on  a Bakhvalov-type  mesh for a singularly perturbed problem with two parameters
\tnoteref{funding} }

\tnotetext[funding]{
This research is supported by National Natural Science Foundation of China (11771257, 11601251).
}

\author[label1] {Jin Zhang\corref{cor1}}
\author[label1] {Yanhui Lv \fnref{cor2}}
\cortext[cor1] {Corresponding author: jinzhangalex@hotmail.com }
\fntext[cor2] {Email: yanhuilv@hotmail.com }
\address[label1]{School of Mathematics and Statistics,  Shandong Normal University, 
Jinan 250014,  China}

\begin{abstract}
In this paper, the linear finite element method  on a Bakhvalov-type  mesh is applied to a  singularly perturbed problem with two parameters. The solution of the problem exists two exponential boundary layers.  A new interpolation, which is simple in construction and analysis, is introduced for convergence analysis. Furthermore, we find a subtle relationship between the Bakhvalov-type mesh itself and the weaker exponential layer and obtain an interesting result. Finally, we prove  a supercloseness result between the Lagrange interpolation and the numerical solution. Numerical tests confirm our theoretical results.

\end{abstract}


\end{frontmatter}



%
%
%

\section{Introduction}
We consider the two-parameter singularly perturbed boundary value problem
\begin{equation}\label{eq:I-condition-1}
-\varepsilon_1 u''+\varepsilon_2b(x)u'+c(x)u=f(x)\quad \text{in $\Omega:=(0, 1)$}, \quad
u(0)=u(1)=0, 
\end{equation}
with
\begin{align}
b(x)\ge \lambda>0, \quad c(x)\ge \beta>0\quad \text{on $\bar{\Omega}$}, \label{eq:I-condition-2}\\
c(x)-\frac{1}{2}\varepsilon_2b'(x)\ge \gamma>0\quad \text{on $\bar{\Omega}$}\label{eq:I-condition-3}, 
\end{align}
where $b$,  $c$,  $f$ are sufficiently smooth functions on $\bar{\Omega}=[0, 1]$ and $\lambda$,  $\beta$,  $\gamma$ are constants.  Here we are interested in the case $0<\varepsilon_1, \varepsilon_2\ll 1$ (see \cite{OMalley:1967-Two-parameter}). Thus the problem \eqref{eq:I-condition-1}
is singularly perturbed. Moreover,  the conditions \eqref{eq:I-condition-2} and \eqref{eq:I-condition-3} ensure that  there exists a unique solution $u\in H_0^1(\Omega)$,   which is characterized by boundary layers at $x = 0$ and $x = 1$.

It is found that the values of $\varepsilon_1$ and $\varepsilon_2$ strongly affect the properties of boundary layers. When $\varepsilon_1\ll \varepsilon_2^2\ll 1$,  the boundary layer in
the vicinity of $x=1$ is stronger than the boundary layer at $x=0$, as is similar to the  reaction-convection-diffusion type problem. When $\varepsilon_2^2\ll \varepsilon_1\ll 1$, the widths of the layers around $x=0$ and $x=1$ are almost the same, as holds for the reaction-diffusion case. And when 
 $\varepsilon_2=1$, the problem becomes a single parameter problem
 and exhibits boundary layer only at $x = 0$.


For singular perturbation problems, it is well known that standard numerical methods do not work. Therefore, layer-adapted  meshes represented by Shishkin-type meshes and Bakhvalov-type meshes (see \cite{Linb:2010-Layer}) are explored to solve these problems. However, 
different from  Shishkin-type meshes, Bakhvalov-type meshes brings great difficulties in convergence analysis (see \cite[\S 2]{Roos-2006-Error}), so there are few articles on this kind of meshes at present. But, the existence of Bakhvalov-type meshes is of great significance, because 
 numerical results of Bakhvalov-type meshes are often significantly better than those of Shishkin-type meshes.


In recent years, the study of supercloseness has attracted more and more attention, since the supercloseness results play an important role in the posterior error analysis  as well as for
improved approximations of the solution(see \cite{Li1Wheeler2:2000-Uniform,Li:2001-Convergence,Roos1Lins2:2001-Dradient,Styn1Tobi2:2003-SDFEM,Fran1Lins2:2008-Superconvergence}). However,  only the supercloseness study  on Shishkin-type meshes exists  at present for two-parameter singular perturbation problems  (see \cite{Zhang:2003-Finite,Teof1Zari2:2009-Superconvergence,Teof1Brdar2Fran3Zari4:2018-SDFEM}), while the research on Bakhvalov-type meshes seems to be in a blank state due  to the lack of powerful technical tools.



In this paper, we obtain the supercloseness results of the linear finite element method on a Bakhvalov-type mesh for the two-parameter singularly perturbed problem. First, we use a new interpolation, which has a simple structure and is easy for convergence analysis on Bakhvalov-type meshes. Second, we find a subtle relationship between the Bakhvalov-type mesh itself and the weaker exponential layer. Using this relationship, we get an interesting  result. In addition, 
we would like to emphasize that in the analysis we found that the difficulties of supercloseness were mainly caused by the weak exponential layer and the regular part, rather than being controlled by the strong exponential layer.
This may shed light to the subsequent analysis on higher-dimensional problems.

Our paper is organized as follows. In Section 2 we present properties of the exact solution as well as its derivatives,  then introduce a Bakhvalov-type  mesh and its properties,  and define the finite element method  of our problem. 
In Section 3 we define an interpolation $\Pi u$ and prove interpolation error bounds for the Lagrange interpolation. Section 4 contains error bounds of $u^I-u^N$ and the main result on the uniform error bound in the energy norm. 

In this paper,  the value of $C$ is independent of the perturbation parameters $\varepsilon_1$,  $\varepsilon_2$ and
the number $N$. For a set $D\subset\mathbb{R}$,  we use the standard notation for Banach spaces $L^p(D)$,  Sobolev spaces
$W^{k,  p}(D)$,  $H^k(D)=W^{k,  2}(D)$. For $p=2$,  we have  the norm $\Vert\cdot\Vert_D$ for $\Vert\cdot\Vert_{L^2(D)}$, 
and the seminorm $\vert\cdot\vert_{1, D}$ for $\Vert\cdot\Vert_{H^1(D)}$. When $D=\Omega$ we drop the subscript $D$ from the  notation for simplicity.
 
\section{Regularity,  Bakhvalov-type  mesh and finite element method}
\subsection{Characteristic of Solution }\label{sec:cof}
According to \cite[\S 2]{Brda1Zari2:2016-singularly},  we describe the characteristics of the solution to problem \eqref{eq:I-condition-1},  which are usually needed for error estimations in the case of singularly perturbed problems.

In order to describe the layers in the solution of \eqref{eq:I-condition-1},  the characteristic equation is introduced
$$-\varepsilon_1g^2(x)+\varepsilon_2b(x)g(x)+c(x)=0.$$
It has two real solutions $g_0(x)<0$ and $g_1(x)>0$ which characterize the layers. Let
$$\mu_0=-\max\limits_{0\leq x\leq1} g_0(x),  \quad \mu_1=\min\limits_{0\leq x\leq1} g_1(x), $$
where $$\mu_0=-g_0(x_*)=\frac{-\varepsilon_2b_*+\sqrt{\varepsilon_2^2b_*^2+4\varepsilon_1c_*}}{2\varepsilon_1}, $$
with $b_*=b(x_*)$ and $c_*=c(x_*)$ for some $x_*\in [0, 1]$.
The later analysis needs the following  properties of $\mu_0$  and $\mu_1$ (see \cite{Tefa1Roos2:2007-elliptic})
\begin{align}
&\mu_0\leq\mu_1, \qquad \qquad \max\{\mu_0^{-1},  \varepsilon_1\mu_1\}\leq C(\varepsilon_2+\varepsilon_1^{\frac{1}{2}}), \label{cos:conclusions-1}\\
&\varepsilon_2\mu_0\leq \lambda^{-1}\Vert c\Vert_{L^\infty(\Omega)}, \qquad\qquad\varepsilon_2(\varepsilon_1\mu_1)^{-\frac{1}{2}}\label{cos:conclusions-2}
\leq C\varepsilon_2^{\frac{1}{2}}.
\end{align}
The values $\mu_0$ and $\mu_1$ determine the decay of the boundary layers.   The following lemma from \cite{Linb:2010-Layer} provides more details on the behavior of the solution of \eqref{eq:I-condition-1} and its derivatives.
\begin{lemma}\label{lem1:cof}
Let $b$,  $c$,  $f\in C^q(\bar{\Omega})$ for some $q\geq 1$ and let $p,  \kappa\in(0, 1)$ be arbitrary. Assume that
$$q\Vert b'\Vert_{L^\infty(\Omega)}\varepsilon_2\leq\kappa(1-p).$$
Then
$$\vert u^{(k)}(x)\vert\leq C(1+\mu_0^ke^{-p\mu_0x}+\mu_1^ke^{-p\mu_1(1-x)})\qquad x\in\Omega, $$
for $0\leq k\leq q$.
\end{lemma}
Then, from \cite{Linss:2001-necessity} the solution $u$ of \eqref{eq:I-condition-1}  has the representation
\begin{equation}\label{eq:u-decomposition}
u=S+E_0+E_1, 
\end{equation}
with
\begin{equation}\label{lem1:cof-1}
\vert S^{(k)}\vert \leq C,  \qquad\vert E_0^{(k)}\vert\leq C\mu_0^ke^{-p\mu_0x},  \qquad\vert E_1^{(k)}\vert\leq C\mu_1^ke^{-p\mu_1(1-x)}, 
\end{equation}
for $x\in\Omega$ and $0\leq k\leq q$.

\subsection{Bakhvalov-type  mesh and  finite element method }
Let $N\in \mathbb{N},  N\ge 16$,  be divisible by 4, and
\begin{equation}\label{eq:break point}
\sigma_j=\frac{\tau}{p\mu_j}\ln \mu_j\leq\frac{1}{4}\quad j=0,  1, 
\end{equation}
where $\tau\geq1$ is a user-chosen parameter and $p\in(0, 1)$ is the parameter from Lemma \ref{lem1:cof}. We set $\sigma_0$ and $1-\sigma_1$ as Bakhvalov-type  mesh's transition points.
Transition points are the points where the mesh changes from fine to coarse and viceversa. Set $\Omega_0=(0,  \sigma_0)$,   $\Omega_e=(\sigma_0,  1-\sigma_1)$ and $\Omega_1=(1-\sigma_1,  1)$.
Then define a mesh on $\bar\Omega$ such that it is equidistant on $\bar\Omega_e$ with $N/2$ mesh subintervals,  and gradually divided on $\bar\Omega_0$ and $\bar\Omega_1$ with $N/4$ mesh subintervals.
Similar to \cite{Brda1Zari2:2016-singularly}, the mesh points $x_i$,  $i=0,  1,  \cdots,  N$,  are defined by
\begin{equation}\label{eq:mesh-points}
x_i=\begin{cases}
 \frac{\tau}{p\mu_0}\varphi_0(t_i)\qquad\qquad i=0,  1,  \cdots,  \frac{N}{4}, \\
 \sigma_0+2(t_i-\frac{1}{4})(1-\sigma_0-\sigma_1)\qquad i=\frac{N}{4}, \frac{N}{4}+1, \cdots, \frac{3N}{4}, \\
 1-\frac{\tau}{p\mu_1}\varphi_1(t_i)\qquad\qquad i=\frac{3N}{4},  \frac{3N}{4}+1,  \cdots,  N, 
\end{cases}
\end{equation}
where $t_i=\frac{i}{N}$,  $i=0,  1,  \cdots,  N$ and
$$
\varphi_0(t)=-\ln(1-4(1-\mu_0^{-1})t), \quad
\varphi_1(t)=-\ln(1-4(1-\mu_1^{-1})(1-t)).
$$
The mesh generating functions $\varphi_0$ and $\varphi_1$ are piecewise continuously differentiable and satisfy the following conditions
\begin{align*}
\varphi_0(0)=0, \quad \varphi_0\left(\frac{1}{4}\right)=\ln\mu_0, \\
\varphi_1\left(\frac{3}{4}\right)=\ln\mu_1, \quad \varphi_1(1)=0.
\end{align*}

The weak form of problem \eqref{eq:I-condition-1} is to find $u\in H^1_0(\Omega)$ such that
\begin{equation}\label{eq:weak form-1d}
a(u,  v)=(f,  v) \quad \forall v\in H^1_0(\Omega), 
\end{equation}
where
\begin{equation*}\label{eq:bilinear form}
a(u, v):=\varepsilon_1 ( u',  v')+\varepsilon_2(b u',  v)+(cu, v)  \quad  u,  v\in H^1_0(\Omega), 
\end{equation*}
and $(\cdot, \cdot)$ denotes the standard scalar product in $L^2(\Omega)$.

Set $I_i:=[x_i, x_{i+1}]$ with $i=0,  1,  \cdots,  N-1$. Define the linear finite element space on the Bakhvalov-type  mesh \eqref{eq:mesh-points}
\begin{equation*}\label{eq:VN}
V^{N}=\{w\in C(\bar{\Omega}):\; w(0)=w(1)=0, \;
 \text{$w|_{I_i}\in P_1(I_i)$ for $i=0, \ldots, N-1$ } \}, 
\end{equation*}
with the standard basis $\{\theta_i\}_{i=0}^N$ of hat functions,  i.e.,  $\theta_i(x_j)=\delta_{ij},  i,  j=0, 1, \cdots, N$.
The finite element method for \eqref{eq:weak form-1d} reads as: Find $u^N\in V^{N}$ such that
\begin{equation}\label{eq:FE-1d}
a(u^N, v^N)=(f, v^N) \quad \forall v^N\in V^N.
\end{equation}
The energy norm associated with $a(\cdot, \cdot)$ is defined by
\begin{equation*}\label{eq:energy norm}
\Vert v \Vert_{ E}^2:=
\varepsilon_1 \vert  v \vert^2_1+ \Vert v \Vert^2  \quad \forall v\in H^1(\Omega).
\end{equation*}
Using \eqref{eq:I-condition-3},  it is easy to see that one has the coercivity
\begin{equation*}\label{eq:coercivity}
a(v^N, v^N) \ge \alpha \Vert v^N \Vert_{E}^2\quad \text{for all $v^N\in V^N$}, 
\end{equation*}
with $\alpha=\min (1, \gamma)$. It follows that $u^N$ is well defined by \eqref{eq:FE-1d} (see \cite{Bren1Scot2:2008-mathematical} and references therein). Hence,  the problem \eqref{eq:FE-1d} has a unique solution.

%
%
%
\subsection{Some preliminary results}\label{sec:mesh}
In this subsection,  we present some preliminary results on the Bakhvalov-type  mesh \eqref{eq:mesh-points}.

Since $1\ll\mu_0\leq\mu_1$,  throughout the paper we assume
\begin{equation}\label{eq:N to solutions}
\mu_1^{-1}\leq\mu_0^{-1}\leq N^{-1}, 
\end{equation}
which is not a restriction in practice.
\begin{remark}\label{eq:remark1}
Without assumption \eqref{eq:N to solutions},  there will be $\mu_0^{-1}>N^{-1}$ or $\mu_1^{-1}>N^{-1}$. We might as well take $\mu_0^{-1}>N^{-1}$,  then the boundary layer at $x=0$ can be captured by standard triangulation and it is unnecessary for us to use finer meshes in the layer area.

\end{remark}

Similar to  Lemma 2  in \cite{Zhan1Liu2:2020-Optimal},  one has the following two lemmas.
\begin{lemma}\label{lem:mesh size}
The mesh sizes $h_i:=x_{i+1}-x_i,  i=0,  1,  \cdots, N-1$,  satisfy
\begin{align*}
&h_0\leq h_1\leq \cdots\leq h_{\frac{N}{4}-2},\\
&\frac{\tau}{4p}\mu_0^{-1}\leq h_{\frac{N}{4}-2}\leq \frac{\tau}{p}\mu_0^{-1}, \\
&\frac{\tau}{2p}\mu_0^{-1}\leq h_{\frac{N}{4}-1}\leq \frac{4\tau}{p}N^{-1}, \\
&N^{-1}\leq h_i\leq 2N^{-1} \qquad\frac{N}{4}\leq i\leq \frac{3N}{4}-1,\\
&\frac{\tau}{2p}\mu_1^{-1}\leq h_{\frac{3N}{4}}\leq \frac{4\tau}{p}N^{-1}, \\
&\frac{\tau}{4p}\mu_1^{-1}\leq h_{\frac{3N}{4}+1}\leq \frac{\tau}{p}\mu_1^{-1}, \\
&h_{\frac{3N}{4}+1}\geq h_{\frac{3N}{4}+2}\geq \cdots\geq h_{N-1}.
\end{align*}
\end{lemma}

\begin{lemma}\label{lem:Error-condition}
On the Bakhvalov-type mesh \eqref{eq:mesh-points},  we obtain  for $0\leq m\leq\tau$
\begin{align}
&h_i^me^{-p\mu_0x_i}\leq C\mu_0^{-m}N^{-m} \qquad  i=0,1,\cdots,\frac{N}{4}-2, \label{lem:e-E0}\\
&h_i^me^{-p\mu_1(1-x_{i+1})}\leq C\mu_1^{-m}N^{-m} \qquad i=\frac{3N}{4}+1,\frac{N}{4}+2,\cdots, N-1.\label{lem:e-E1}
\end{align}
\end{lemma}

Here we introduce a very interesting lemma as follows, which plays a key role in the later analysis.
\begin{lemma}\label{Ti}
Assume $\tau\geq \frac{5}{2}$. On the Bakhvalov-type mesh \eqref{eq:mesh-points}, one has
\begin{equation*}
\max\limits_{i=0,1,\cdots,\frac{N}{4}-3}\{(h_{i+1}-h_i)e^{-p\mu_0x_{i+1}}\}\leq  C\mu_0^{-1}N^{-2}.
\end{equation*}
\end{lemma}
\begin{proof}
Set $T_i=(h_{i+1}-h_i)e^{-p\mu_0x_{i+1}}$ for $i=0,1,\cdots,\frac{N}{4}-3.$ 
From \eqref{eq:mesh-points} we obtain 
\begin{equation*}\label{eq:Ti-definition}
\begin{aligned}
T_i&=(h_{i+1}-h_i)e^{-p\mu_0x_{i+1}}\\
&=\frac{\tau}{p\mu_0}\ln \left(\frac{(1-\xi(i+1)N^{-1})^2}{[1-\xi(i+2)N^{-1}][1-\xi iN^{-1}]}\right)\cdot(1-\xi(i+1)N^{-1})^{\tau}\\
&=C\mu_0^{-1}\ln \left(1+\frac{\xi^2N^{-2}}{[1-\xi(i+2)N^{-1}][1-\xi iN^{-1}]}\right)\cdot(1-\xi(i+1)N^{-1})^{\tau},
\end{aligned}
\end{equation*}
where $\xi=4(1-\mu_0^{-1}).$
Thus  we have
\begin{equation}\label{eq:T0}
\begin{aligned}
T_0&=C\mu_0^{-1}\ln \left(1+\frac{\xi^2N^{-2}}{(1-2\xi N^{-1})}\right)\cdot(1-\xi N^{-1})^{\tau}\\
&\leq C\mu_0^{-1}\ln \left(1+\frac{\xi^2N^{-2}}{(1-2\xi N^{-1})}\right)\\
&\leq C\mu_0^{-1}\frac{\xi^2N^{-2}}{1-2\xi N^{-1}}\\
&\leq C\mu_0^{-1}N^{-2},
\end{aligned}
\end{equation}
and from \eqref{eq:N to solutions} we can get
\begin{equation}\label{eq:B-N/4-3}
\begin{aligned}
&T_{\frac{N}{4}-3}\\
&=C\mu_0^{-1}\ln \left(1+\frac{\xi^2N^{-2}}{[1-\xi(\frac{1}{4}-N^{-1})][1-\xi(\frac{1}{4}-3N^{-1})]}\right)\cdot(1-\xi(\frac{1}{4}-2N^{-1}))^{\tau}\\
&=C\mu_0^{-1}\ln \left(1+\frac{\xi^2N^{-2}}{[\mu_0^{-1}+\xi N^{-1})(\mu_0^{-1}+3\xi N^{-1}]}\right)\cdot(\mu_0^{-1}+\xi N^{-1})^{\tau}\\
&\leq C\mu_0^{-1}\ln \left(1+\frac{1}{3}\right)\cdot(N^{-1}+\xi N^{-1})^{\tau}\\
&\leq C\mu_0^{-1}N^{-\tau}\leq C\mu_0^{-1}N^{-\frac{5}{2}}.
\end{aligned}
\end{equation}

Next, we set $T(s)=C\mu_0^{-1}\ln \left(1+\frac{\xi^2N^{-2}}{[1-\xi(s+2)N^{-1}][1-\xi sN^{-1}]}\right)\cdot(1-\xi(s+1)N^{-1})^{\tau}$ for $s\in[0,\frac{N}{4}-3]$, 
then we take the derivative of $T(s)$ to analyze the monotonicity of $T(s)$:
\begin{equation*}\label{eq:Ti-derivation}
T'(s)=M\left(2Q(s)-\tau\ln (1+Q(s))\right),
\end{equation*}
where $M= \xi N^{-1}(1-\xi(s+1) N^{-1})^{\tau-1}$ and
\begin{align*}
Q(s)&=\frac{\xi^2N^{-2}}{(1-\xi(s+2)N^{-1})(1-\xi sN^{-1})}\\
&\in\left[\frac{\xi^2N^{-2}}{1-2\xi N^{-1}},\frac{\xi^2N^{-2}}{(\mu_0^{-1}+\xi N^{-1})(\mu_0^{-1}+3\xi N^{-1})}\right].
\end{align*}
Obviously  $M>0$. So we just analyze the positive and negative properties of $2Q(s)-\tau\ln (1+Q(s))$
and
 set 
$$F(t)=2t-\tau\ln (1+t)\quad\quad t\in[0,+\infty).$$
After a brief analysis, we found that when $\tau\geq\frac{5}{2}$, $F(0)=0$, $F(t)<0$  in $[0,t_*)$ and $F(t)\ge0$ in $[t_*,+\infty]$ for some $t_*\in(0,+\infty)$.

Therefore, when $\tau\geq \frac{5}{2}$, $T(s)$ decreases monotonously as $s$ increases, or decreases first and then increases, or increases monotonously. To sum up, we can get $\max\limits_{s\in[0,\frac{N}{4}-3]}\{T(s)\}= \max\{T(0),T(\frac{N}{4}-3)\}$ . And then combine \eqref{eq:T0} and \eqref{eq:B-N/4-3} to get our conclusion.
\end{proof}

\section{Interpolation errors}
Now we introduce a new interpolation defined in \cite{Zhan1Liu2:2020-Optimal} for our uniform convergence and briefly describe the structure of this interpolation.

For any $v\in C^0(\bar{\Omega})$ its Lagrange interpolation $v^I$ on the  Bakhvalov-type  mesh \eqref{eq:mesh-points} is defined by
$$v^I=\sum_{i=0}^{N}v(x_i)\theta_i(x),$$
where $\theta_i(x) $ is the piecewise linear  polynomial satisfying  the conditions $\theta_i(x_j)=\delta_{ij}$ for $i,j=0,\ldots,N$. Here $\delta_{ij}$ is the Kronecker symbol.
We define the interpolation $\Pi u$ to the solution $u$ by
\begin{equation*}\label{eq:Interpolation-u}
\Pi u=S^I+E_0^I+\pi E_1, 
\end{equation*}
where $S^I$ and $E_0^I$ are the Lagrange interpolation to $S$ and $E_0$, respectively. And
\begin{align}
\pi E_1=&\sum_{i=0}^{\frac{3N}{4}}E_1(x_i)\theta_i(x)+\sum_{i=\frac{3N}{4}+2}^NE_1(x_i)\theta_i(x).\label{eq:Interpolation-E1}
\end{align}
Define
\begin{equation*}\label{eq:Interpolation-PE1}
(P E_1)(x)=E_1(x_{\frac{3N}{4}+1})\theta_{\frac{3N}{4}+1}(x), 
\end{equation*}
and clearly we have  $\pi E_1,  \Pi u\in V^N$ 
and
\begin{align}
&(\pi E_1)(x)=E_1^I-(PE_1)(x), \label{eq:P_jE_j}\\
&\Pi u=u^I-(P E_1)(x), \label{eq:Interpolation-PEj-Ej}\\ 
&\pi E_1\vert_{[x_0, x_{\frac{3N}{4}}]\cup[x_{\frac{3N}{4}+2}, x_N]}
=E_1^I\vert_{[x_0, x_{\frac{3N}{4}}]\cup[x_{\frac{3N}{4}+2}, x_N]},\label{eq:Interpolation-PE1-E1}
\end{align}
where $E_1^I$ and $u^I$ are Lagrange interpolation of $E_1$ and $u^I$, respectively.

Similar to \cite[(38)]{Zhan1Liu2:2020-Optimal},  the following lemma can be proved.
\begin{lemma}\label{eq:PE-norm}
Assume $\tau\geq \frac{5}{2}$. Then one has
\begin{equation*}\label{eq:PE-energy-norm}
\Vert (PE_1)(x)\Vert_E
 \leq CN^{-\frac{5}{2}}(\varepsilon_2+\varepsilon_1^{\frac{1}{2}})^{\frac{1}{2}}+CN^{-3}.
\end{equation*}
\end{lemma}
\begin{remark}\label{eq:remark2}
It can be seen from the definition that the new interpolation introduced by us has the advantages of simple structure and easy analysis compared with quasi-interpolation(see \cite{Brda1Zari2:2016-singularly}). Moreover, our interpolation can be easily extended to general problems, such as higher-order finite element methods and higher-dimensional singular perturbation problems.
\end{remark}
Next we will give some results of Lagrange interpolation.
Standard interpolation theories yield
\begin{equation}\label{eq:Interpolation-error}
\Vert v-v^I\Vert_{W^{l, q}(I_i)}\leq Ch_i^{2-l+\frac{1}{q}-\frac{1}{p}}\vert v\vert_{W^{2, p}(I_i)}, 
\end{equation}
for all $v\in W^{2, p}(I_i),  i=0,  1,  \cdots,  N-1, $ where $l=0, 1$ and $1\leq p, q\leq \infty$.
\begin{lemma}\label{eq:Interpolation-L2}
Assume $\tau\geq \frac{5}{2}$. On the Bakhvalov-type mesh \eqref{eq:mesh-points},  one has
\begin{align*}
&\Vert E_j-E_j^I\Vert\leq CN^{-\frac{5}{2}},\quad j=0,1,\\
&\Vert S-S^I\Vert+\Vert u-u^I\Vert\leq CN^{-2}.
\end{align*}
\end{lemma}
\begin{proof}
For $0\leq i\leq \frac{N}{4}-2$,  we use \eqref{eq:Interpolation-error},  \eqref{lem1:cof-1} and \eqref{lem:e-E0} with $m=\frac{5}{2}$ to obtain
$$
\begin{aligned}
\sum_{i=0}^{\frac{N}{4}-2}\Vert E_0-E_0^I\Vert_{I_i}^2&\leq C\sum_{i=0}^{\frac{N}{4}-2}h_i^4\Vert E_0''\Vert_{I_i}^2
\leq C\sum_{i=0}^{\frac{N}{4}-2}h_i^5\Vert E_0''\Vert_{L^\infty(I_i)}^2\\
&\leq C\sum_{i=0}^{\frac{N}{4}-2}h_i^5\mu_0^4e^{-2p\mu_0x_i}=C\sum_{i=0}^{\frac{N}{4}-2}\mu_0^4(h_i^{\frac{5}{2}}e^{-p\mu_0x_i})^2\\
&\leq C\mu_0^{-1}N^{-4}.
\end{aligned}\\$$
For $\frac{N}{4}-1\leq i\leq N-1$,   we use the triangle inequality,  Lemma \ref{lem:mesh size} and \eqref{lem1:cof-1} to get
$$
\begin{aligned}
\sum_{i=\frac{N}{4}-1}^{N-1}\Vert E_0-E_0^I\Vert_{I_i}^2&\leq C\sum_{i=\frac{N}{4}-1}^{N-1}h_i\Vert E_0-E_0^I\Vert_{L^\infty(I_i)}^2
\leq C\sum_{i=\frac{N}{4}-1}^{N-1}h_i\Vert E_0\Vert_{L^\infty(I_i)}^2\\
&\leq C\sum_{i=\frac{N}{4}-1}^{N-1}h_ie^{-2p\mu_0x_{\frac{N}{4}-1}}
\leq C\sum_{i=\frac{N}{4}-1}^{N-1}N^{-1}N^{-2\tau}\\
&\leq CN^{-5}.
\end{aligned}\\$$
From \eqref{eq:N to solutions},  we have
$$\Vert E_0-E_0^I\Vert \leq C(\mu_0^{-1}N^{-4}+N^{-5})^{\frac{1}{2}}\leq CN^{-\frac{5}{2}}.$$

By similar arguments we can get the interpolation error of $E_1$.
Standard arguments yield
$
\Vert S-S^I\Vert \leq CN^{-2}$.
Thus we are done.
\end{proof}
\begin{lemma}\label{eq:Interpolation-H1}
Assume $\tau\geq \frac{5}{2}$.  On the Bakhvalov-type  mesh \eqref{eq:mesh-points},   one has
\begin{align*}
&\vert E_j-E_j^I\vert_{1}\leq C\mu_j^{\frac{1}{2}}N^{-1}  \qquad j=0,   1, \\
&\vert S-S^I\vert_{1}\leq CN^{-1}.
\end{align*}
\end{lemma}
\begin{proof}
For $0\leq i\leq \frac{N}{4}-2$,    from \eqref{eq:Interpolation-error},   \eqref{lem1:cof-1} and \eqref{lem:e-E0} with $m=\frac{3}{2}$,  we have
$$
\begin{aligned}
\sum_{i=0}^{\frac{N}{4}-2}\vert E_0-E_0^I\vert_{1,   I_i}^2
&\leq C\sum_{i=0}^{\frac{N}{4}-2}h_i^2\Vert E_0''\Vert_{I_i}^2
\leq C\sum_{i=0}^{\frac{N}{4}-2}h_i^3\Vert E_0''\Vert_{L^\infty(I_i)}^2\\
&\leq C\sum_{i=0}^{\frac{N}{4}-2}h_i^3\mu_0^4e^{-2p\mu_0x_i}
=C\sum_{i=0}^{\frac{N}{4}-2}\mu_0^4(h_i^{\frac{3}{2}}e^{-p\mu_0x_i})^2\\
&\leq C\sum_{i=0}^{\frac{N}{4}-2}\mu_0^4\mu_0^{-3}N^{-3}\leq C\mu_0N^{-2}.
\end{aligned}\\$$
For $i=\frac{N}{4}-1,$ from the triangle inequality, the inverse inequality, \eqref{lem1:cof-1} and Lemma \ref{lem:mesh size},   one has
$$\begin{aligned}
\vert E_0-E_0^I\vert_{1,I_{\frac{N}{4}-1}}^2&\leq C\vert E_0\vert_{1,I_{\frac{N}{4}-1}}^2+C\vert E_0^I\vert_{1,I_{\frac{N}{4}-1}}^2\\
&\leq C\int_{x_{\frac{N}{4}-1}}^{x_{\frac{N}{4}}}(E_0')^2dx+h_{\frac{N}{4}-1}^{-2}\Vert E_0^I\Vert_{I_{\frac{N}{4}-1}}^2\\
&\leq C\int_{x_{\frac{N}{4}-1}}^{x_{\frac{N}{4}}}\mu_0^2e^{-2p\mu_0x}dx+Ch_{\frac{N}{4}-1}^{-1}\Vert E_0\Vert_{L^\infty(I_{\frac{N}{4}-1})}^2\\
&\leq C\mu_0e^{-2p\mu_0x_{\frac{N}{4}-1}}+C\mu_0e^{-2p\mu_0x_{\frac{N}{4}-1}}\\
&\leq C\mu_0N^{-2\tau}\leq C\mu_0N^{-5}.
\end{aligned}$$
For $\frac{N}{4}\leq i\leq N-1,$  from \eqref{eq:Interpolation-error}, \eqref{lem1:cof-1} and Lemma \ref{lem:mesh size}, one has 
$$
\begin{aligned}
\sum_{i=\frac{N}{4}}^{N-1}\vert E_0-E_0^I\vert_{1, I_i}^2
&\leq C\sum_{i=\frac{N}{4}}^{N-1}h_i^2\Vert E_0''\Vert_{I_i}^2
\leq C\sum_{i=\frac{N}{4}}^{N-1}h_i^3\Vert E_0''\Vert_{L^\infty(I_i)}^2\\
&\leq C\sum_{i=\frac{N}{4}}^{N-1}h_i^3\mu_0^4e^{-2p\mu_0x_i}
\leq C\sum_{i=\frac{N}{4}}^{N-1}N^{-3}\mu_0^4e^{-2p\mu_0x_{\frac{N}{4}}}\\
&\leq CN^{-2}\mu_0^4\mu_0^{-2\tau}
\leq C\mu_0^{-1}N^{-2}.
\end{aligned}$$

Finally,   we conclude
$$\vert E_0-E_0^I\vert_{1}\leq C(\mu_0N^{-2}+\mu_0N^{-5}+C\mu_0^{-1}N^{-2})^{\frac{1}{2}}\leq C\mu_0^{\frac{1}{2}}N^{-1}.$$

The estimate for $\vert E_1-E_1^I\vert_1$ can be similarly derived. Standard arguments yield the bound of 
$\vert S-S^I\vert_{1}$.
\end{proof}

From \eqref{cos:conclusions-1}, Lemmas \ref{eq:Interpolation-L2} and   \ref{eq:Interpolation-H1},  we  obtain the following theorem easily. 
\begin{theorem}\label{eq:Interpolation-energy-norm}
Assume $\tau\geq \frac{5}{2}$. Then one has
\begin{equation*}\label{eq:interpolation-energy-norm}
\Vert u-u^I\Vert_E\leq C(\varepsilon_2+\varepsilon_1^{\frac{1}{2}})^{\frac{1}{2}}N^{-1}+CN^{-2}.
\end{equation*}
\end{theorem}

\section{Uniform convergence}
Introduce $\chi:=\Pi u-u^N$. From \eqref{eq:coercivity},  the Galerkin orthogonality,  \eqref{eq:u-decomposition},  \eqref{eq:Interpolation-PEj-Ej} and integration by parts for $\displaystyle\int_0^1b(\pi E_1-E_1)'\chi\mr{d}x, $   one has
\begin{equation}\label{eq:uniform-convergence-1}
\begin{aligned}
&\alpha\Vert \chi\Vert_E^2\leq a(\chi,   \chi)=a(\Pi u-u,   \chi)\\
&=\varepsilon_1\int_0^1(u^I-u)'\chi'dx-\varepsilon_1\int_0^1(P E_1)'\chi'dx\\
&+\varepsilon_2\int_0^1b(S^I-S)'\chi dx+\varepsilon_2\int_0^1b( E_0^I-E_0)'\chi dx\\
&-\varepsilon_2\int_0^1b(\pi E_1-E_1)\chi'dx-\varepsilon_2\int_0^1b'(\pi E_1-E_1)\chi dx\\
&+\int_0^1c(u^I-u)\chi dx-\int_0^1c(P E_1)\chi dx\\
&=\uppercase\expandafter{\romannumeral1}+\uppercase\expandafter{\romannumeral2}
+\uppercase\expandafter{\romannumeral3}+\uppercase\expandafter{\romannumeral4}
+\uppercase\expandafter{\romannumeral5}+\uppercase\expandafter{\romannumeral6}
+\uppercase\expandafter{\romannumeral7}+\uppercase\expandafter{\romannumeral8}.
\end{aligned}
\end{equation}
In the following we will analyze each term in the right-hand side of \eqref{eq:uniform-convergence-1}.

From integration by parts we  first prove
\begin{equation}\label{eq:I}
\uppercase\expandafter{\romannumeral1}=0.
\end{equation}

From H\"{o}lder inequalities, \eqref{eq:P_jE_j},  Lemmas  \ref{eq:PE-norm} and \ref{eq:Interpolation-L2},   we obtain
\begin{equation}\label{eq:II+VI+VII+VIII}
\begin{aligned}
&\vert(\uppercase\expandafter{\romannumeral2}+\uppercase\expandafter{\romannumeral8})+
(\uppercase\expandafter{\romannumeral6}+\uppercase\expandafter{\romannumeral7})\vert\\
\leq & C\Vert PE_1\Vert _E\Vert \chi\Vert _E
+C(\varepsilon_2\Vert E_1-E_1^I\Vert
+\varepsilon_2\Vert PE_1\Vert+\Vert u-u^I\Vert)\Vert\chi\Vert\\
\leq &CN^{-2}\Vert\chi\Vert_E.
\end{aligned}
\end{equation}

Imitating the analysis in Lemma 4 in \cite{Zhan1Liu2:2021-Supercloseness}.  One has 
\begin{equation}\label{eq:III}
\vert\uppercase\expandafter{\romannumeral3}\vert\leq C\varepsilon_2^{\frac{1}{2}}N^{-2}\ln^{\frac{1}{2}}N\Vert \chi\Vert_E.
\end{equation}

In the following two lemmas we will give the estimates of  \uppercase\expandafter{\romannumeral4} and \uppercase\expandafter{\romannumeral5}.
\begin{lemma}\label{eq:IV-A}
Assume $\tau\geq \frac{5}{2}$.  On the Bakhvalov-type  mesh \eqref{eq:mesh-points},   one has
\begin{equation*}
\vert\uppercase\expandafter{\romannumeral4}\vert\leq  C(\varepsilon_2^{\frac{1}{2}}N^{-2}\ln^{\frac{1}{2}}N+N^{-3})\Vert\chi\Vert_E.
\end{equation*}
\end{lemma}
\begin{proof}
For the sake of analysis, the term $\int_0^1b( E_0^I-E_0)'\chi dx$ is separated into three parts as follows:
\begin{align*}
\int_0^1b( E_0^I-E_0)'\chi dx
&=\int_{x_0}^{x_{\frac{N}{4}-1}}b(E_0^I-E_0)'\chi dx
+\int_{x_{\frac{N}{4}-1}}^{x_{\frac{N}{4}}}b(E_0^I-E_0)'\chi dx\\
&+\int_{x_{\frac{N}{4}}}^{x_N}b(E_0^I-E_0)'\chi dx\\
&=:\Delta_1+\Delta_2+\Delta_3.
\end{align*} 
For $\Delta_1$,  we first use the integral identity (see \cite[(5.14)]{Linb:2010-Layer}) to get
\begin{equation}\label{eq:integral identity}
\begin{aligned}
\int_{I_i}(E_0-E_0^I)'\chi dx&=\frac{1}{6}\int_{I_i}E_0'''(J_i^2)'\chi' dx -\frac{1}{3}\left(\frac{h_i}{2}\right)^2\int_{I_i}E_0'''\chi dx \\
&+\frac{1}{3}\left(\frac{h_i}{2}\right)^2E_0''\chi\vert_{x_{i}}^{x_{i+1}},
\end{aligned}
\end{equation}
where $J_i(x)=\frac{1}{2}(x-x_{i})(x-x_{i+1}).$
Then from \eqref{eq:integral identity} and H\"{o}lder inequalities we can obtain
\begin{equation*}
\begin{aligned}
\vert\Delta_1\vert&\leq C\sum_{i=0}^{\frac{N}{4}-2}h_{i}^3\Vert E_0'''\Vert_{I_i}\Vert\chi'\Vert_{I_i}
+C\sum_{i=0}^{\frac{N}{4}-2}h_{i}^2\Vert E_0'''\Vert_{I_i}\Vert\chi\Vert_{I_i}\\
&+Ch_0^2\vert E_0''(0)\vert\vert\chi(0)\vert+C\sum_{i=0}^{\frac{N}{4}-3}(h_{i+1}-h_i)(h_{i+1}+h_i)\vert E_0''(x_{i+1})\vert\vert\chi(x_{i+1})\vert\\
&+Ch_{\frac{N}{4}-2}^2\vert E_0''(x_{\frac{N}{4}-1})\vert\vert\chi(x_{\frac{N}{4}-1})\vert\\
&=:A_1+A_2+A_3+A_4+A_5.
\end{aligned}
\end{equation*}
Using  the inverse inequality, \eqref{lem1:cof-1} and \eqref{lem:e-E0} with $m=\frac{5}{2}$ to obtain
\begin{equation}\label{eq:A1}
\begin{aligned}
A_1&\leq C\sum_{i=0}^{\frac{N}{4}-2}h_i^3h_i^{-\frac{1}{2}}\Vert E_0'''\Vert_{L^{\infty}(I_i)}\Vert\chi\Vert_{I_i}
\leq C\sum_{i=0}^{\frac{N}{4}-2}h_i^{\frac{5}{2}}\mu_0^3e^{-p\mu_0x_i}\Vert\chi\Vert_{I_i}\\
&\leq C\sum_{i=0}^{\frac{N}{4}-2}\mu_0^{\frac{1}{2}}N^{-\frac{5}{2}}\Vert\chi\Vert_{I_i}
\leq C\mu_0^{\frac{1}{2}}\left(\sum_{i=0}^{\frac{N}{4}-2}N^{-5}\right)^{\frac{1}{2}}\left(\sum_{i=0}^{\frac{N}{4}-2}\Vert\chi\Vert_{I_i}^2\right)^{\frac{1}{2}}\\
&\leq C\mu_0^{\frac{1}{2}}N^{-2}\Vert\chi\Vert_{[x_0,x_{\frac{N}{4}-1}]}.
\end{aligned}
\end{equation}
And to use \eqref{lem1:cof-1} and \eqref{lem:e-E0} with $m=\frac{5}{2}$ again to get 
\begin{equation}\label{eq:A2}
\begin{aligned}
A_2&\leq C\sum_{i=0}^{\frac{N}{4}-2}h_i^2h_i^{\frac{1}{2}}\Vert E_0'''\Vert_{L^{\infty}(I_i)}\Vert\chi\Vert_{I_i}
\leq C\sum_{i=0}^{\frac{N}{4}-2}h_i^{\frac{5}{2}}\mu_0^3e^{-p\mu_0x_i}\Vert\chi\Vert_{I_i}\\
&\leq C\sum_{i=0}^{\frac{N}{4}-2}\mu_0^{\frac{1}{2}}N^{-\frac{5}{2}}\Vert\chi\Vert_{I_i}
\leq C\mu_0^{\frac{1}{2}}\left(\sum_{i=0}^{\frac{N}{4}-2}N^{-5}\right)^{\frac{1}{2}}\left(\sum_{i=0}^{\frac{N}{4}-2}\Vert\chi\Vert_{I_i}^2\right)^{\frac{1}{2}}\\
&\leq C\mu_0^{\frac{1}{2}}N^{-2}\Vert\chi\Vert_{[x_0,x_{\frac{N}{4}-1}]}.
\end{aligned}
\end{equation}
Obviously we have $\chi\in V^N$, so one has
\begin{equation}\label{eq:A3}
A_3=0.
\end{equation}
From \eqref{lem1:cof-1}, Lemmas \ref{lem:mesh size}, \ref{Ti}, and the inverse inequality  we have
\begin{equation}\label{eq:A4}
\begin{aligned}
A_4&\leq C\sum_{i=0}^{\frac{N}{4}-3}(h_{i+1}-h_i)(h_{i+1}+h_i)\mu_0^2e^{-p\mu_0x_{i+1}}\vert\chi(x_{i+1})\vert\\
&\leq C\mu_0N^{-2}\sum_{i=0}^{\frac{N}{4}-3}(h_{i+1}+h_i)\vert\chi(x_{i+1})\vert
\leq C\mu_0N^{-2}\sum_{i=0}^{\frac{N}{4}-3}h_{i+1}\Vert\chi\Vert_{L^{\infty}(I_{i+1})}\\
&\leq C\mu_0N^{-2}\sum_{i=0}^{\frac{N}{4}-3}h_{i+1}h_{i+1}^{-\frac{1}{2}}\Vert\chi\Vert_{I_{i+1}}\\
&\leq C\mu_0N^{-2}\left(\sum_{i=0}^{\frac{N}{4}-3}h_{i+1}\right)^{\frac{1}{2}}\left(\sum_{i=0}^{\frac{N}{4}-3}\Vert\chi\Vert_{I_{i+1}}^2\right)^{\frac{1}{2}}\\
&\leq C\mu_0N^{-2}\mu_0^{-\frac{1}{2}}\ln^{\frac{1}{2}}N\Vert\chi\Vert_{[x_0,x_{\frac{N}{4}-1}]}\\
&\leq C\mu_0^{\frac{1}{2}}N^{-2}\ln^{\frac{1}{2}}N\Vert\chi\Vert_{[x_0,x_{\frac{N}{4}-1}]}.
\end{aligned}
\end{equation}
From \eqref{lem1:cof-1}, Lemma \ref{lem:mesh size} and the inverse inequality we can obtain
\begin{equation}\label{eq:A5}
\begin{aligned}
A_5&\leq Ch_{\frac{N}{4}-2}^2\mu_0^2e^{-p\mu_0x_{\frac{N}{4}-1}}\vert\chi(x_{\frac{N}{4}-1})\vert
\leq CN^{-\tau}\Vert\chi\Vert_{L^{\infty}(I_{\frac{N}{4}-2})}\\
&\leq CN^{-\frac{5}{2}}h_{\frac{N}{4}-2}^{-\frac{1}{2}}\Vert\chi\Vert_{I_{\frac{N}{4}-2}}\\
&\leq C\mu_0^{\frac{1}{2}}N^{-\frac{5}{2}}\Vert\chi\Vert_{I_{\frac{N}{4}-2}}.
\end{aligned}
\end{equation}
Thus from \eqref{eq:A1}--\eqref{eq:A5} we have 
\begin{equation}\label{eq:Delta1}
\vert\Delta_1\vert\leq C\mu_0^{\frac{1}{2}}N^{-2}\ln^{\frac{1}{2}}N\Vert\chi\Vert_{E,[x_0,x_{\frac{N}{4}-1}]}.
\end{equation}
 For $\Delta_2$, using H\"{o}lder inequalities, the triangle inequality, the inverse inequality, Lemma \ref{lem:mesh size} and \eqref{lem1:cof-1}  yield
\begin{equation}\label{eq:Delta2}
\begin{aligned}
\Delta_2&\leq C\Vert(E_0^I)'\Vert_{I_{\frac{N}{4}-1}}\Vert\chi\Vert_{I_{\frac{N}{4}-1}}+C\Vert E_0'\Vert_{I_{\frac{N}{4}-1}}\Vert\chi\Vert_{I_{\frac{N}{4}-1}}\\
&\leq Ch_{\frac{N}{4}-1}^{-\frac{1}{2}}\Vert E_0^I\Vert_{L^{\infty}(I_{\frac{N}{4}-1})}\Vert\chi\Vert_{I_{\frac{N}{4}-1}}
+Ch_{\frac{N}{4}-1}^{\frac{1}{2}}\Vert E_0'\Vert_{L^{\infty}(I_{\frac{N}{4}-1})}\Vert\chi\Vert_{I_{\frac{N}{4}-1}}\\
&\leq C\mu_0^{\frac{1}{2}}e^{-p\mu_0x_{\frac{N}{4}-1}}\Vert\chi\Vert_{I_{\frac{N}{4}-1}}
+CN^{-\frac{1}{2}}\mu_0e^{-p\mu_0x_{\frac{N}{4}-1}}\Vert\chi\Vert_{I_{\frac{N}{4}-1}}\\
&\leq C\mu_0^{\frac{1}{2}}N^{-\tau}\Vert\chi\Vert_{I_{\frac{N}{4}-1}}+C\mu_0N^{-(\frac{1}{2}+\tau)}\Vert\chi\Vert_{I_{\frac{N}{4}-1}}\\
&\leq C\mu_0^{\frac{1}{2}}N^{-\frac{5}{2}}\Vert\chi\Vert_{I_{\frac{N}{4}-1}}+C\mu_0N^{-3}\Vert\chi\Vert_{I_{\frac{N}{4}-1}}.
\end{aligned}
\end{equation}
For $\Delta_3$, using H\"{o}lder inequalities, \eqref{eq:Interpolation-error},  \eqref{lem1:cof-1} and Lemma \ref{lem:mesh size}  to obtain
\begin{equation}\label{eq:Delta3}
\begin{aligned}
\Delta_3&\leq C\sum_{i=\frac{N}{4}}^{N-1}\Vert(E_0^I-E_0)'\Vert_{I_i}\Vert\chi\Vert_{I_i}
\leq C\sum_{i=\frac{N}{4}}^{N-1}h_i^{\frac{3}{2}}\Vert E_0''\Vert_{L^{\infty}(I_i)}\Vert\chi\Vert_{I_i}\\
&\leq C\sum_{i=\frac{N}{4}}^{N-1}h_i^{\frac{3}{2}}\mu_0^2e^{-p\mu_0x_{\frac{N}{4}}}\Vert\chi\Vert_{I_i} 
\leq C\sum_{i=\frac{N}{4}}^{N-1}h_i^{\frac{3}{2}}\mu_0^2\mu_0^{-\tau}\Vert\chi\Vert_{I_i}\\ 
&\leq C\sum_{i=\frac{N}{4}}^{N-1}N^{-\frac{3}{2}}\mu_0^{-\frac{1}{2}}\Vert\chi\Vert_{I_i}
\leq C\mu_0^{-\frac{1}{2}}\left(\sum_{i=\frac{N}{4}}^{N-1}N^{-3}\right)^{\frac{1}{2}}\left(\sum_{i=\frac{N}{4}}^{N-1}\Vert\chi\Vert_{I_i}^2\right)^{\frac{1}{2}}\\
&\leq C\mu_0^{-\frac{1}{2}}N^{-1}\Vert\chi\Vert_{[x_{\frac{N}{4}},x_N]}
\end{aligned}
\end{equation}

Therefore from \eqref{eq:Delta1}--\eqref{eq:Delta3} and \eqref{eq:N to solutions}  we have
\begin{equation*}
\vert\int_0^1b( E_0^I-E_0)'\chi dx\vert\leq C(\mu_0^{\frac{1}{2}}N^{-2}\ln^{\frac{1}{2}}N+\mu_0N^{-3})\Vert\chi\Vert.
\end{equation*}
Combine that with \eqref{cos:conclusions-2} we get 
\begin{equation*}\label{eq:IV-B}
\vert\varepsilon_2\int_0^1b( E_0^I-E_0)'\chi dx\vert\leq C(\varepsilon_2^{\frac{1}{2}}N^{-2}\ln^{\frac{1}{2}}N+N^{-3})\Vert\chi\Vert_E.
\end{equation*}
\end{proof}


\begin{lemma}\label{eq:V-A}
Assume $\tau\geq \frac{5}{2}$ 
 and let $\pi E_1$ be defined in \eqref{eq:Interpolation-E1}. On the Bakhvalov-type  mesh \eqref{eq:mesh-points},   one has
\begin{equation*}\label{eq:V-B}
\vert\uppercase\expandafter{\romannumeral7}\vert
\leq C\varepsilon_2^{\frac{1}{2}}N^{-2}\Vert\chi\Vert _E.
\end{equation*}
\end{lemma}
\begin{proof}
According to \eqref{eq:Interpolation-PE1-E1},  the term $\int_0^1b(\pi E_1-E_1)\chi'dx$ is separated into three parts as follows:
$$\begin{aligned}
&\int_0^1b(\pi E_1-E_1)\chi'dx\\
&=\int_{x_0}^{x_{\frac{3N}{4}}}b(E_1^I-E_1)\chi'dx
+\int_{x_{\frac{3N}{4}}}^{x_{\frac{3N}{4}+2}}b(\pi E_1- E_1)\chi'dx
+\int_{x_{\frac{3N}{4}+2}}^{x_N}b(E_1^I-E_1)\chi'dx\\
&=\Gamma_1+\Gamma_2+\Gamma_3.
\end{aligned}\\$$
For $\Gamma_1$, from  H\"{o}lder inequalities, \eqref{eq:Interpolation-error}, \eqref{lem1:cof-1} and Lemma \ref{lem:mesh size}, we have
\begin{equation}\label{eq:VII-1}
\begin{aligned}
\vert\Gamma_1\vert &
\leq C\sum_{i=0}^{\frac{3N}{4}-1}\Vert E_0^I-E_0\Vert_{I_i}\Vert \chi'\Vert_{I_i}
\leq C\sum_{i=0}^{\frac{3N}{4}-1}h_i^{\frac{5}{2}}\Vert E_1''\Vert_{L^{\infty}(I_i)}\Vert \chi'\Vert_{I_i}\\
&\leq C\sum_{i=0}^{\frac{3N}{4}-1}h_i^{\frac{5}{2}}\mu_1^2e^{-p\mu_1(1-x_{\frac{N}{4}})}\Vert \chi'\Vert_{I_i}
\leq C\sum_{i=0}^{\frac{3N}{4}-1}N^{-\frac{5}{2}}\mu_1^2\mu_1^{-\tau}\Vert \chi'\Vert_{I_i}\\
&\leq C\sum_{i=0}^{\frac{3N}{4}-1}N^{-\frac{5}{2}}\mu_1^{-\frac{1}{2}}\Vert \chi'\Vert_{I_i}
\leq C\mu_1^{-\frac{1}{2}}\left(\sum_{i=0}^{\frac{3N}{4}-1}N^{-5}\right)^{\frac{1}{2}}
\left(\sum_{i=0}^{\frac{3N}{4}-1}\Vert \chi'\Vert_{I_i}^2\right)^{\frac{1}{2}}\\
&\leq C\mu_1^{-\frac{1}{2}}N^{-2}\Vert \chi'\Vert_{[x_0, x_{\frac{3N}{4}}]}\\
&\leq C\varepsilon_1^{-\frac{1}{2}}\mu_1^{-\frac{1}{2}}N^{-2}\Vert \chi\Vert_{E, [x_0, x_{\frac{3N}{4}}]}.
\end{aligned}
\end{equation}
Now we analyze the term $\Gamma_2$. Note $\pi E_1=E_1^I-(PE_1)(x)$ on $[x_{\frac{3N}{4}+1}, x_{\frac{3N}{4}+2}]$ and
from H\"{o}lder inequalities, \eqref{eq:Interpolation-error}, \eqref{lem1:cof-1}, \eqref{lem:e-E1} with $m=\frac{5}{2}$ and Lemma \ref{lem:mesh size},  
one has
\begin{equation}\label{eq:VII-2-1}
\begin{aligned}
&\vert\int_{x_{\frac{3N}{4}+1}}^{x_{\frac{3N}{4}+2}}b(\pi E_1-E_1)\chi'dx\vert
\leq C\int_{x_{\frac{3N}{4}+1}}^{x_{\frac{3N}{4}+2}}\vert E_1^I-E_1\vert\vert\chi'\vert dx\\
&+C \vert E_1(x_{\frac{3N}{4}+1})\vert\int_{x_{\frac{3N}{4}+1}}^{x_{\frac{3N}{4}+2}}\vert\theta_{\frac{3N}{4}+1}(x)\chi'\vert dx\\
&\leq C(\Vert E_1^I-E_1\Vert_{I_{\frac{3N}{4}+1}}
+\vert E_1(x_{\frac{3N}{4}+1})\vert h_{\frac{3N}{4}+1}^{\frac{1}{2}})\Vert \chi'\Vert_{I_{\frac{3N}{4}+1}}\\
&\leq C(h_{\frac{3N}{4}+1}^{\frac{5}{2}}\Vert E_1''\Vert_{L^{\infty}(I_{\frac{3N}{4}+1})}
+N^{-\tau}\mu_1^{-\frac{1}{2}})\Vert \chi'\Vert_{I_{\frac{3N}{4}+1}}\\
&\leq C(h_{\frac{3N}{4}+1}^{\frac{5}{2}}\mu_1^2e^{-p\mu_1(1-x_{\frac{3N}{4}+2})}
+N^{-\tau}\mu_1^{-\frac{1}{2}})\Vert \chi'\Vert_{I_{\frac{3N}{4}+1}}\\
&\leq C(\mu_1^{-\frac{1}{2}}N^{-\frac{5}{2}}
+N^{-\frac{5}{2}}\mu_1^{-\frac{1}{2}})\Vert \chi'\Vert_{I_{\frac{3N}{4}+1}}\\
&\leq C\varepsilon_1^{-\frac{1}{2}}\mu_1^{-\frac{1}{2}}N^{-\frac{5}{2}}\Vert \chi\Vert_{E, I_{\frac{3N}{4}+1}}.
\end{aligned}
\end{equation}
On $[x_{\frac{3N}{4}}, x_{\frac{3N}{4}+1}]$, we have $\pi E_1=E_1(x_{\frac{3N}{4}})\theta_{\frac{3N}{4}}(x)$ from \eqref{eq:Interpolation-E1}. Thus
from H\"{o}lder inequalities,  \eqref{lem1:cof-1} and Lemma \ref{lem:mesh size}, we have
\begin{equation}\label{eq:VII-2-2}
\begin{aligned}
&\vert\int_{x_{\frac{3N}{4}}}^{x_{\frac{3N}{4}+1}}b(\pi E_1-E_1)\chi'dx\vert\\
&\leq C\vert E_1(x_{\frac{3N}{4}})\vert\int_{x_{\frac{3N}{4}}}^{x_{\frac{3N}{4}+1}}\vert\theta_{\frac{3N}{4}}(x)\vert\vert\chi'\vert dx
+C \int_{x_{\frac{3N}{4}}}^{x_{\frac{3N}{4}+1}}\vert E_1\vert\vert\chi'\vert dx\\
&\leq C(\vert E_1(x_{\frac{3N}{4}})\vert\Vert\theta_{\frac{3N}{4}}(x)\Vert_{I_{\frac{3N}{4}}}
+\Vert E_1\Vert_{I_{\frac{3N}{4}}})\Vert\chi'\Vert_{I_{\frac{3N}{4}}}\\
&\leq C(\mu_1^{-\tau}h_{\frac{3N}{4}}^{\frac{1}{2}}+\mu_1^{-\frac{1}{2}}N^{-\tau})\Vert\chi'\Vert_{I_{\frac{3N}{4}}}\\
&\leq C(\mu_1^{-\frac{5}{2}}N^{-\frac{1}{2}}+\mu_1^{-\frac{1}{2}}N^{-\frac{5}{2}})\Vert\chi'\Vert_{I_{\frac{3N}{4}}}\\
&\leq C\varepsilon_1^{-\frac{1}{2}}(\mu_1^{-\frac{5}{2}}N^{-\frac{1}{2}}+\mu_1^{-\frac{1}{2}}N^{-\frac{5}{2}})\Vert\chi\Vert_{E, I_{\frac{3N}{4}}}.
\end{aligned}
\end{equation}
Therefore, from \eqref{eq:VII-2-1} and \eqref{eq:VII-2-2} we can obtain
\begin{equation}\label{eq:VII-2}
\vert\Gamma_2\vert\leq C\varepsilon_1^{-\frac{1}{2}}(\mu_1^{-\frac{5}{2}}N^{-\frac{1}{2}}+\mu_1^{-\frac{1}{2}}N^{-\frac{5}{2}})\Vert\chi\Vert_{E, [x_{\frac{3N}{4}}, x_{\frac{3N}{4}+2}]}.
\end{equation}
For $\Gamma_3$, from H\"{o}lder inequalities, \eqref{eq:Interpolation-error},  \eqref{lem:e-E1} with $m=\frac{5}{2}$, we have
\begin{equation}\label{eq:VII-3}
\begin{aligned}
\vert \Gamma_3\vert&
\leq C\sum_{i=\frac{3N}{4}+2}^{N-1}\Vert E_1^I-E_1\Vert_{I_i}\Vert \chi'\Vert_{I_i}
\leq C\sum_{i=\frac{3N}{4}+2}^{N-1}h_i^{\frac{5}{2}}\Vert E_1''\Vert_{L^{\infty}(I_i)}\Vert \chi'\Vert_{I_i}\\
&\leq C\sum_{i=\frac{3N}{4}+2}^{N-1}h_i^{\frac{5}{2}}\mu_1^2e^{-p\mu_1(1-x_{i+1})}\Vert \chi'\Vert_{I_i}
\leq C\sum_{i=\frac{3N}{4}+2}^{N-1}\mu_1^{-\frac{1}{2}}N^{-\frac{5}{2}}\Vert \chi'\Vert_{I_i}\\
&\leq C\mu_1^{-\frac{1}{2}}\left(\sum_{i=\frac{3N}{4}+2}^{N-1}N^{-5}\right)^{\frac{1}{2}}
\left(\sum_{i=\frac{3N}{4}+2}^{N-1}\Vert \chi'\Vert_{I_i}^2\right)^{\frac{1}{2}}\\
&\leq C\mu_1^{-\frac{1}{2}}N^{-2}\Vert \chi'\Vert_{[x_{\frac{3N}{4}+2}, x_N]}\\
&\leq C\varepsilon_1^{-\frac{1}{2}}\mu_1^{-\frac{1}{2}}N^{-2}\Vert \chi\Vert_{E, [x_{\frac{3N}{4}+2}, x_N]}.
\end{aligned}
\end{equation}

From \eqref{eq:VII-1},  \eqref{eq:VII-2},  \eqref{eq:VII-3}, \eqref{eq:N to solutions} and  \eqref{cos:conclusions-2}  we are done.
\end{proof}


Now we are in a position to present our main results.
\begin{theorem}\label{the:u^I-u^N}
Assume $\tau\geq \frac{5}{2}$. Let the mesh $\{x_i\}$ be the Bakhvalov-type mesh \eqref{eq:mesh-points}. Then we have
\begin{align*}
&\Vert u^I-u^N\Vert_E\leq C\varepsilon_2^{\frac{1}{2}}N^{-2}\ln^{\frac{1}{2}}N+CN^{-2}.
\end{align*}
\end{theorem}
\begin{proof}
From  \eqref{eq:I}, \eqref{eq:II+VI+VII+VIII}, \eqref{eq:III} and  Lemmas \ref{eq:IV-A}--\ref{eq:V-A},   we can easily obtain 
\begin{equation}\label{eq:chi}
\Vert\chi\Vert_E\leq C\varepsilon_2^{\frac{1}{2}}N^{-2}\ln^{\frac{1}{2}}N+CN^{-2}.
\end{equation}
Recalling that $\Vert u^I-u^N\Vert_E\leq C\Vert\chi\Vert_E+C\Vert PE_1\Vert_E$, then our conclusion can be drawn from \eqref{eq:chi} and Lemma \ref{eq:PE-norm}.
\end{proof}

Same to Theorem 4 in \cite{Brda1Zari2:2016-singularly},  we prove the  optimal order of uniform convergence in the following theorem. 
\begin{theorem}\label{the:u-u^N}
Assume $\tau\geq \frac{5}{2}$. Let the mesh $\{x_i\}$ be the Bakhvalov-type mesh \eqref{eq:mesh-points}. Then we have
\begin{equation*}\label{eq:u-u^N}
\Vert u-u^N\Vert _E\leq
C(\varepsilon_2+\varepsilon_1^{\frac{1}{2}})^{\frac{1}{2}}N^{-1}+CN^{-2}.
\end{equation*}
\end{theorem}
\begin{proof}
From Theorem \ref{eq:Interpolation-energy-norm} and Theorem \ref{the:u^I-u^N} we can easily prove the result.
\end{proof}
\begin{remark}
From   Theorems \ref{the:u^I-u^N} and \ref{the:u-u^N},  we can see   Theorem  \ref{the:u^I-u^N} presents a supercloseness result when $\varepsilon_2+\varepsilon_1^{\frac{1}{2}}< CN^{-2}$.
\end{remark}
\section{Numerical experimentation}
In this section, we present numerical examples to verify our theoretical findings. 
We focus on numerical experiments on  supercloseness results. 
Here,  a test problem in \cite{Brda1Zari2:2016-singularly} is introduced, which is adequate for our purpose,
\begin{equation*}\label{eq:problem}
-\varepsilon_1u''+\varepsilon_2u'+u=\cos \pi x \quad x\in \Omega,\quad u(0)=u(1)=0,
\end{equation*}
with the exact solution
\begin{equation*}\label{eq:exact solution}
u(x)=a\cos\pi x+b\sin\pi x+Ae^{-\mu_0x}+Be^{-\mu_1(1-x)},
\end{equation*}
where 
\begin{align*}
&a=\frac{ \varepsilon_1\pi^2+1}{\varepsilon_2^2\pi^2+(\varepsilon_1\pi^2+1)^2},\quad b=\frac{ \varepsilon_2\pi}{\varepsilon_2^2\pi^2+(\varepsilon_1\pi^2+1)^2} ,\\
&A=-a\frac{1+e^{-\mu_1}}{1-e^{-\mu_0-\mu_1}},\quad B=a\frac{1+e^{-\mu_0}}{1-e^{-\mu_0-\mu_1}},\quad\mu_{0,1}=\frac{\mp\varepsilon_2+\sqrt{\varepsilon_2^2+4\varepsilon_1}}{2\varepsilon_1}.
\end{align*}

%
%

In order that \eqref{eq:break point} and \eqref{eq:N to solutions} hold true, considering $16\le N \le 2048$ in our numerical experiments and we take
\begin{align*}
&0<\varepsilon_1\le 10^{-7},\quad 0<\varepsilon_2\le 10^{-4},\\
&p=0.5,\quad \tau=2.
\end{align*}
We also consider $\varepsilon_1=1,10^{-2},10^{-4},10^{-6}$ or $\varepsilon_2=1$ to check the according numerical performances.
For a fixed $\varepsilon_1$, $\varepsilon_2$ and $N$, we consider the errors 
\begin{equation*} 
e^N=\Vert u-u^N \Vert_E,\quad e^N_{I}=\Vert u^I-u^N \Vert_E,
\end{equation*}  
and the corresponding rates of convergence
\begin{equation*}
p^N=\frac{ \ln e^N-\ln e^{2N} }{\ln 2},\quad p^N_I=\frac{ \ln e^N_{I}-\ln e^{2N}_{I} }{\ln 2}.
\end{equation*}

In Tables \ref{table:1}--\ref{table:6} we present the errors $e^N$, $e^N_{I}$ and the convergence orders $p^N$, $p^N_I$ for $\varepsilon_1=1,10^{-2},10^{-6},10^{-8},10^{-10}$ and $N=16,\ldots,2048$ in the case of $\varepsilon_2=1,10^{-4},10^{-8}$. 


The data show that no matter $\varepsilon_2$ is taken as a large value or a small value, there exists supercloseness results  as long as $\varepsilon_1$ and $\varepsilon_2$ satisfy proper conditions. 
And we can see in some cases, some orders of $\Vert u-u^N\Vert_E$ are  closer  ones of $\Vert u^I-u^N\Vert_E$ such as the last two columns in Tables \ref{table:3}--\ref{table:6} , so the supercloseness  is very weak.
These phenomena support  Theorem \ref{the:u^I-u^N}.



 
\begin{table}[h]
\caption{$\Vert u-u^N\Vert_E$ in the case of $\varepsilon_2=1$}
\footnotesize
\resizebox{90mm}{30mm}{
\begin{tabular*}{\textwidth}{@{\extracolsep{\fill}} c| cccccccccc cc}
\cline{1-13}{}
             \diagbox{$N$}{$\varepsilon_1$}    &$10^{0}$&$10^{0}$    &$10^{-2}$ & $10^{-2}$  & $10^{-4}$ & $10^{-4}$ & $10^{-6}$ & $10^{-6}$ & $10^{-8}$ & $10^{-8}$& $10^{-10}$& $10^{-10}$\\
\cline{1-13}
    &$e^N$&$p^N$&$e^N$&$p^N$&$e^N$&$p^N$&$e^N$&$p^N$&$e^N$&$p^N$&$e^N$&$p^N$\\ 
\cline{1-13}
             $16$       & 0.12E-1  & 1.00    & 0.41E-1  & 1.06  & 0.35E-1   & 0.96  & 0.35E-1   & 0.95    & 0.35E-1   & 0.95    & 0.35E-1   & 0.95\\
             $32$       & 0.62E-2  & 1.00    & 0.20E-1  & 1.01  & 0.18E-1   & 1.00  & 0.18E-1   & 1.00    & 0.18E-1   & 1.00    & 0.18E-1   & 1.00\\
             $64$       & 0.31E-2  & 1.00    & 0.98E-2  & 1.00  & 0.90E-2   & 1.00  & 0.90E-2   & 1.00    & 0.90E-2   & 1.00    & 0.90E-2   & 1.00\\
             $128$      & 0.16E-2  & 1.00    & 0.49E-2  & 1.00  & 0.45E-2   & 1.00  & 0.45E-2   & 1.00    & 0.45E-2   & 1.00    & 0.45E-2   & 1.00\\
             $256$      & 0.78E-3  & 1.00    & 0.25E-2  & 1.00  & 0.22E-2   & 1.00  & 0.22E-2   & 1.00    & 0.22E-2   & 1.00    & 0.22E-2   & 1.00\\
             $512$      & 0.39E-3  & 1.00    & 0.12E-2  & 1.00  & 0.11E-2   & 1.00  & 0.11E-2   & 1.00    & 0.11E-2   & 1.00    & 0.11E-2   & 1.00\\
             $1024$     & 0.19E-3  & 1.00    & 0.62E-3  & 1.00  & 0.56E-3   & 1.00  & 0.56E-3   & 1.00    & 0.56E-3   & 1.00    & 0.56E-3   & 1.00\\
             $2048$     & 0.97E-4  & 1.00    & 0.31E-3  & 1.00  & 0.28E-3   & 1.00  & 0.28E-3   & 1.00    & 0.28E-3   & 1.00    & 0.28E-3   & 1.00\\
             $4096$     & 0.49E-4  & ---     & 0.15E-3  & ---   & 0.14E-3   & ---   & 0.14E-3   & ---     & 0.14E-3   & ---    & 0.14E-3   & ---\\
\cline{1-13}
\end{tabular*}}
\label{table:1}
\end{table}
\begin{table}[h]
\caption{$\Vert u^I-u^N\Vert_E$ in the case of $\varepsilon_2=1$}
\footnotesize
\resizebox{90mm}{30mm}{
\begin{tabular*}{\textwidth}{@{\extracolsep{\fill}} c| cccccccccc cc}
\cline{1-13}{}
             \diagbox{$N$}{$\varepsilon_1$}    &$10^{0}$&$10^{0}$    &$10^{-2}$ & $10^{-2}$  & $10^{-4}$ & $10^{-4}$ & $10^{-6}$ & $10^{-6}$ & $10^{-8}$ & $10^{-8}$& $10^{-10}$& $10^{-10}$\\
\cline{1-13}
    &$e^N_I$&$p^N_I$&$e^N_I$&$p^N_I$&$e^N_I$&$p^N_I$&$e^N_I$&$p^N_I$&$e^N_I$&$p^N_I$&$e^N_I$&$p^N_I$\\ 
\cline{1-13}
             $16$       & 0.21E-3  & 2.00    & 0.14E-1  & 2.04  & 0.81E-2   & 1.84  & 0.79E-2   & 1.82    & 0.79E-2   & 1.82    & 0.79E-2   & 1.82\\
             $32$       & 0.53E-4  & 2.00    & 0.34E-2  & 1.88  & 0.23E-2   & 1.99  & 0.22E-2   & 1.98    & 0.22E-2   & 1.98    & 0.22E-2   & 1.98\\
             $64$       & 0.13E-4  & 2.00    & 0.91E-3  & 1.85  & 0.57E-3   & 2.00  & 0.57E-3   & 2.00    & 0.57E-3   & 2.00    & 0.57E-3   & 2.00\\
             $128$      & 0.33E-5  & 2.00    & 0.25E-3  & 1.87  & 0.14E-3   & 2.00  & 0.14E-3   & 2.00    & 0.14E-3   & 2.00    & 0.14E-3   & 2.00\\
             $256$      & 0.83E-6  & 2.00    & 0.69E-4  & 1.92  & 0.36E-4   & 2.00  & 0.35E-4   & 2.00    & 0.35E-4   & 2.00    & 0.35E-4   & 2.00\\
             $512$      & 0.21E-6  & 2.00    & 0.18E-4  & 1.97  & 0.89E-5   & 2.00  & 0.89E-5   & 2.00    & 0.89E-5   & 2.00    & 0.89E-5   & 2.00\\
             $1024$     & 0.52E-7  & 2.00    & 0.47E-5  & 1.99  & 0.22E-5   & 2.00  & 0.22E-5   & 2.00    & 0.22E-5   & 2.00    & 0.22E-5   & 2.00\\
             $2048$     & 0.13E-7  & 2.00    & 0.12E-5  & 2.00  & 0.56E-6   & 2.00  & 0.55E-6   & 2.00    & 0.55E-6   & 2.00    & 0.55E-6   & 2.00\\
             $4096$     & 0.33E-8  & ---     & 0.30E-6  & ---   & 0.14E-6   & ---   & 0.14E-6   & ---     & 0.14E-6   & ---     & 0.14E-6   & ---\\
\cline{1-13}
\end{tabular*}}
\label{table:2}
\end{table}
\begin{table}[h]
\caption{$\Vert u-u^N\Vert_E$ in the case of $\varepsilon_2=10^{-4}$}
\footnotesize
\resizebox{90mm}{30mm}{
\begin{tabular*}{\textwidth}{@{\extracolsep{\fill}} c| cccccccccc cc}
\cline{1-13}{}
             \diagbox{$N$}{$\varepsilon_1$}    &$10^{0}$&$10^{0}$    &$10^{-2}$ & $10^{-2}$  & $10^{-4}$ & $10^{-4}$ & $10^{-6}$ & $10^{-6}$ & $10^{-8}$ & $10^{-8}$& $10^{-10}$& $10^{-10}$\\
\cline{1-13}
    &$e^N$&$p^N$&$e^N$&$p^N$&$e^N$&$p^N$&$e^N$&$p^N$&$e^N$&$p^N$&$e^N$&$p^N$\\
\cline{1-13}
             $16$       & 0.25E-1  & 1.00    & 0.62E-1  & 1.00  & 0.43E-1   & 1.03  & 0.14E-1   & 1.13    & 0.89E-2   & 1.66    & 0.13E-1   & 1.95\\
             $32$       & 0.13E-1  & 1.00    & 0.31E-1  & 1.00  & 0.21E-1   & 1.00  & 0.66E-2   & 1.03    & 0.28E-2   & 1.30    & 0.33E-2   & 1.76\\
             $64$       & 0.63E-2  & 1.00    & 0.15E-1  & 1.00  & 0.10E-1   & 1.01  & 0.32E-2   & 1.01    & 0.11E-2   & 1.08    & 0.98E-3   & 1.36\\
             $128$      & 0.31E-3  & 1.00    & 0.77E-2  & 1.00  & 0.51E-2   & 1.01  & 0.16E-2   & 1.00    & 0.54E-3   & 1.02    & 0.38E-3   & 1.07\\
             $256$      & 0.16E-2  & 1.00    & 0.39E-2  & 1.00  & 0.26E-2   & 1.01  & 0.80E-3   & 1.00    & 0.27E-3   & 1.00    & 0.18E-3   & 1.01\\
             $512$      & 0.78E-3  & 1.00    & 0.19E-3  & 1.00  & 0.13E-2   & 1.00  & 0.40E-3   & 1.00    & 0.13E-3   & 1.00    & 0.90E-4   & 1.00\\
             $1024$     & 0.39E-3  & 1.00    & 0.96E-3  & 1.00  & 0.63E-3   & 1.00  & 0.20E-3   & 1.00    & 0.67E-4   & 1.00    & 0.45E-4   & 1.00\\
             $2048$     & 0.20E-3  & 1.00    & 0.48E-3  & 1.00  & 0.32E-3   & 1.00  & 0.10E-3   & 1.00    & 0.33E-4   & 1.00    & 0.23E-4   & 1.00\\
             $4096$     & 0.98E-4  & ---     & 0.24E-3  & ---   & 0.16E-3   & ---   & 0.50E-4   & ---     & 0.17E-4   & ---     & 0.11E-4   & ---\\
\cline{1-13}
\end{tabular*}}
\label{table:3}
\end{table}
\begin{table}[h]
\caption{$\Vert u^I-u^N\Vert_E$ in the case of $\varepsilon_2=10^{-4}$}
\footnotesize
\resizebox{90mm}{30mm}{
\begin{tabular*}{\textwidth}{@{\extracolsep{\fill}} c| cccccccccc cc}
\cline{1-13}{}
             \diagbox{$N$}{$\varepsilon_1$}    &$10^{0}$&$10^{0}$    &$10^{-2}$ & $10^{-2}$  & $10^{-4}$ & $10^{-4}$ & $10^{-6}$ & $10^{-6}$ & $10^{-8}$ & $10^{-8}$& $10^{-10}$& $10^{-10}$\\
\cline{1-13}
    &$e^N_I$&$p^N_I$&$e^N_I$&$p^N_I$&$e^N_I$&$p^N_I$&$e^N_I$&$p^N_I$&$e^N_I$&$p^N_I$&$e^N_I$&$p^N_I$\\ 
\cline{1-13}
             $16$       & 0.37E-3  & 2.00    & 0.11E-1  & 1.99  & 0.24E-1   & 1.28  & 0.12E-1   & 2.01    & 0.14E-1   & 2.02    & 0.17E-1   & 2.03\\
             $32$       & 0.94E-4  & 2.00    & 0.28E-2  & 2.00  & 0.98E-2   & 1.17  & 0.31E-2   & 2.00    & 0.35E-2   & 2.03    & 0.42E-2   & 2.08\\
             $64$       & 0.23E-4  & 2.00    & 0.71E-3  & 2.00  & 0.44E-2   & 1.30  & 0.77E-3   & 1.92    & 0.86E-3   & 2.02    & 0.10E-2   & 2.14\\
             $128$      & 0.59E-5  & 2.00    & 0.18E-3  & 2.00  & 0.18E-2   & 1.54  & 0.20E-3   & 1.70    & 0.21E-3   & 2.01    & 0.23E-3   & 2.09\\
             $256$      & 0.15E-5  & 2.00    & 0.45E-4  & 2.00  & 0.61E-3   & 1.77  & 0.63E-4   & 1.40    & 0.52E-4   & 2.00    & 0.54E-4   & 2.02\\
             $512$      & 0.37E-6  & 2.00    & 0.11E-4  & 2.00  & 0.18E-3   & 1.91  & 0.24E-4   & 1.31    & 0.13E-4   & 1.99    & 0.13E-4   & 2.00\\
             $1024$     & 0.92E-7  & 2.00    & 0.28E-5  & 2.00  & 0.47E-4   & 1.98  & 0.97E-5   & 1.46    & 0.33E-5   & 1.95    & 0.33E-5   & 1.98\\
             $2048$     & 0.23E-7  & 2.00    & 0.70E-6  & 2.00  & 0.12E-4   & 1.99  & 0.35E-5   & 1.68    & 0.85E-6   & 1.84    & 0.84E-6   & 1.94\\
             $4096$     & 0.57E-8  & ---     & 0.17E-6  & ---   & 0.30E-5   & ---   & 0.11E-5   & ---     & 0.24E-6   & ---     & 0.22E-6   & ---\\
\cline{1-13}
\end{tabular*}}
\label{table:4}
\end{table}
\begin{table}[h]
\caption{$\Vert u-u^N\Vert_E$ in the case of $\varepsilon_2=10^{-8}$}
\footnotesize
\resizebox{90mm}{30mm}{
\begin{tabular*}{\textwidth}{@{\extracolsep{\fill}} c| cccccccccc cc}
\cline{1-13}{}
             \diagbox{$N$}{$\varepsilon_1$}    &$10^{0}$&$10^{0}$    &$10^{-2}$ & $10^{-2}$  & $10^{-4}$ & $10^{-4}$ & $10^{-6}$ & $10^{-6}$ & $10^{-8}$ & $10^{-8}$& $10^{-10}$& $10^{-10}$\\
\cline{1-13}
    &$e^N$&$p^N$&$e^N$&$p^N$&$e^N$&$p^N$&$e^N$&$p^N$&$e^N$&$p^N$&$e^N$&$p^N$\\
\cline{1-13}
             $16$       & 0.25E-1  & 1.00    & 0.62E-1  & 1.00  & 0.43E-1   & 1.03  & 0.14E-1   & 1.13    & 0.74E-2   & 1.55    & 0.64E-2   & 1.94\\
             $32$       & 0.13E-1  & 1.00    & 0.31E-1  & 1.00  & 0.21E-1   & 1.01  & 0.66E-2   & 1.03    & 0.25E-2   & 1.23    & 0.17E-2   & 1.74\\
             $64$       & 0.63E-2  & 1.00    & 0.15E-1  & 1.00  & 0.10E-1   & 1.01  & 0.32E-2   & 1.01    & 0.11E-2   & 1.07    & 0.50E-3   & 1.42\\
             $128$      & 0.31E-2  & 1.00    & 0.77E-2  & 1.00  & 0.51E-2   & 1.01  & 0.16E-2   & 1.00    & 0.51E-3   & 1.02    & 0.19E-3   & 1.16\\
             $256$      & 0.16E-2  & 1.00    & 0.39E-2  & 1.00  & 0.26E-2   & 1.01  & 0.80E-3   & 1.00    & 0.25E-3   & 1.01    & 0.83E-4   & 1.05\\
             $512$      & 0.78E-3  & 1.00    & 0.19E-3  & 1.00  & 0.13E-2   & 1.00  & 0.40E-3   & 1.00    & 0.13E-3   & 1.00    & 0.40E-4   & 1.01\\
             $1024$     & 0.39E-3  & 1.00    & 0.96E-3  & 1.00  & 0.63E-3   & 1.00  & 0.20E-3   & 1.00    & 0.63E-4   & 1.00    & 0.20E-4   & 1.00\\
             $2048$     & 0.20E-3  & 1.00    & 0.48E-3  & 1.00  & 0.32E-3   & 1.00  & 0.10E-3   & 1.00    & 0.32E-4   & 1.00    & 0.10E-4   & 1.00\\
             $4096$     & 0.98E-4  & ---     & 0.24E-3  & ---   & 0.16E-3   & ---   & 0.50E-4   & ---     & 0.16E-4   & ---     & 0.50E-5   & ---\\
\cline{1-13}
\end{tabular*}}
\label{table:5}
\end{table}
\begin{table}[h]
\caption{$\Vert u^I-u^N\Vert_E$ in the case of $\varepsilon_2=10^{-8}$}
\footnotesize
\resizebox{90mm}{30mm}{
\begin{tabular*}{\textwidth}{@{\extracolsep{\fill}} c| cccccccccc cc}
\cline{1-13}{}
             \diagbox{$N$}{$\varepsilon_1$}    &$10^{0}$&$10^{0}$    &$10^{-2}$ & $10^{-2}$  & $10^{-4}$ & $10^{-4}$ & $10^{-6}$ & $10^{-6}$ & $10^{-8}$ & $10^{-8}$& $10^{-10}$& $10^{-10}$\\
\cline{1-13}
    &$e^N_I$&$p^N_I$&$e^N_I$&$p^N_I$&$e^N_I$&$p^N_I$&$e^N_I$&$p^N_I$&$e^N_I$&$p^N_I$&$e^N_I$&$p^N_I$\\ 
\cline{1-13}
             $16$       & 0.37E-3  & 1.99    & 0.11E-1  & 1.99  & 0.24E-1   & 1.28  & 0.12E-1   & 2.01    & 0.13E-1   & 2.00    & 0.14E-1   & 2.00\\
             $32$       & 0.94E-4  & 2.00    & 0.28E-2  & 2.00  & 0.98E-2   & 1.17  & 0.31E-2   & 2.00    & 0.34E-2   & 2.00    & 0.34E-2   & 2.00\\
             $64$       & 0.23E-4  & 2.00    & 0.71E-3  & 2.00  & 0.44E-2   & 1.30  & 0.77E-3   & 1.92    & 0.84E-3   & 1.98    & 0.85E-3   & 2.00\\
             $128$      & 0.59E-5  & 2.00    & 0.18E-3  & 2.00  & 0.18E-2   & 1.54  & 0.20E-3   & 1.70    & 0.21E-3   & 2.01    & 0.21E-3   & 1.99\\
             $256$      & 0.15E-5  & 2.00    & 0.45E-4  & 2.00  & 0.61E-3   & 1.77  & 0.63E-4   & 1.40    & 0.53E-4   & 2.01    & 0.53E-4   & 1.98\\
             $512$      & 0.37E-6  & 2.00    & 0.11E-4  & 2.00  & 0.18E-3   & 1.91  & 0.24E-4   & 1.31    & 0.13E-4   & 2.00    & 0.14E-4   & 2.02\\
             $1024$     & 0.92E-7  & 2.00    & 0.28E-5  & 2.00  & 0.47E-4   & 1.98  & 0.96E-5   & 1.45    & 0.33E-5   & 1.98    & 0.33E-5   & 2.01\\
             $2048$     & 0.23E-7  & 2.00    & 0.70E-6  & 2.00  & 0.12E-4   & 1.99  & 0.35E-5   & 1.68    & 0.83E-6   & 1.92    & 0.83E-6   & 2.00\\
             $4096$     & 0.57E-8  & ---     & 0.17E-6  & ---   & 0.30E-5   & ---   & 0.11E-5   & ---     & 0.22E-6   & ---     & 0.21E-6   & ---\\
\cline{1-13}
\end{tabular*}}
\label{table:6}
\end{table}

\end{document}